\documentclass[10pt]{article}

\newcommand{\N}{\mathbb{N}} 
\newcommand{\R}{\mathbb{R}}

\newcommand{\normal}{\nu}

\renewcommand{\thefootnote}{\fnsymbol{footnote}}

\usepackage[utf8]{inputenc}
\usepackage[english]{babel}
\usepackage{xcolor}
\usepackage{algpseudocode}
\usepackage{algorithm}
\usepackage{array}

\usepackage{graphicx}
\usepackage{caption}
\usepackage{subcaption}
\usepackage{amsfonts}
\usepackage{amsthm,amsmath,amssymb}

\usepackage{authblk}

\usepackage{geometry}
\newcommand{\w}{\mathrm{w}}
\usepackage{amsthm}
\newtheorem{theorem}{Theorem}[section]
\newtheorem{lemma}[theorem]{Lemma}

\newtheorem{problem}{Problem}[section]
\newtheorem{remark}[theorem]{Remark}
\newtheorem{definition}{Definition}[section]
\newtheorem{assumption}{Assumption}[section]

\usepackage{datetime}

\date{\displaydate{date}}

\numberwithin{equation}{section}

\usepackage{scalerel,stackengine}
\stackMath
\newcommand\reallywidehat[1]{%
\savestack{\tmpbox}{\stretchto{%
  \scaleto{%
    \scalerel*[\widthof{\ensuremath{#1}}]{\kern.1pt\mathchar"0362\kern.1pt}%
    {\rule{0ex}{\textheight}}%WIDTH-LIMITED CIRCUMFLEX
  }{\textheight}% 
}{2.4ex}}%
\stackon[-6.9pt]{#1}{\tmpbox}%
}
\begin{document}

\title{\LARGE \bf On an inverse problem with applications in  cardiac electrophysiology  }

\author[1]{Andrea Aspri}
\author[2]{Elena Beretta}
\author[3]{Elisa Francini}
\author[4]{Dario Pierotti}
\author[3]{Sergio Vessella}
\affil[1]{Department of Mathematics, Università degli Studi di Milano}
\affil[2]{Department of Mathematics, NYU Abu Dhabi}
\affil[3]{Department of Mathematics and Computer Science, Università degli Studi di Firenze}
\affil[4]{Department of Mathematics, Politecnico di Milano}

\vspace{5mm}

\date{}

\maketitle

\thispagestyle{plain}
\pagestyle{plain}

\let\thefootnote\relax\footnotetext{


\thanks{}

}

\begin{abstract}
In this paper, we consider the monodomain model of cardiac electrophysiology.
After an analysis of the well-posedness of the forward problem, we show that perfectly insulating regions (modeling ischemic regions in the cardiac tissue) can be uniquely determined by partial boundary measurements of the potential.

\end{abstract}
%{\bf per Andrea: sistemare le citazioni}
\vskip 10truemm

\section{Introduction}

Mathematical models describing the electrical activity of the heart can provide quantitative tools to describe both normal and abnormal heart function. These models often complement imaging techniques, such as computed tomography and magnetic resonance, for diagnostic and therapeutic purposes. In this context, detecting pathological conditions or reconstructing model features, such as tissue conductivities, from potential measurements requires solving an inverse boundary value problem for a nonlinear system of partial differential equations ({\it monodomain model}, \cite{book:pavarino}). 
%\textcolor{red}{Andrea: qual \`e la citazione da inserire qui?}.

It has been observed that, after myocardial infarction and after the healing process is complete, there is a risk of developing lethal ventricular ischemic tachycardia.  This condition is caused by the presence of infarct scars, which have conductivity properties that differ from the surrounding healthy tissue, \cite{book:pavarino}. %\textcolor{red}{Andrea: qual \`e la citazione da inserire qui?}. 
The determination of these regions and their shape from a recording of the electric activity of the heart is fundamental to performing successful ablation for the prevention of tachycardia. 
The authors of \cite{art:bcmp} studied a stationary version of the monodomain model, which led to the investigation of a Neumann problem for a semilinear elliptic equation. They modeled ischemic regions as small conductivity inhomogeneities with lower conductivity compared to the surrounding medium. The presence of the inhomogeneity caused a perturbation in the transmembrane potential, which was described in terms of an asymptotic expansion in a smallness parameter that provided information on the size and shape of the altered region. Based on these findings, a topological gradient method was implemented in \cite{art:BMR} to effectively reconstruct the inhomogeneities from boundary measurements of the potential. A similar analysis was then generalized in \cite{art:BCR20} to the monodomain time-dependent system. 
In \cite{art:BRV}, the authors analyzed a mathematical model that considered inhomogeneities of arbitrary shape and size in a two-dimensional setting. They focused on the difficult issue of reconstructing the conductivity inclusion from boundary measurements, which is a highly nonlinear severely ill-posed inverse problem. In fact, in this case, uniqueness can only be guaranteed if infinitely many measurements of solutions are available and this is clearly not realistic in the application we have in mind.
On the other hand, several studies show that the damaged area can be modeled as an electrical insulator \cite{art:PSIRBF}, \cite{art:RCSRGDRRA},\cite{art:FPLHMDQdB}. 
Mathematically, this leads to modeling the infarcted scar as a perfect
insulator, i.e. as a {\it cavities} in a nonlinear system of reaction-diffusion differential equations.  In this case, the inverse problem is treatable. 
In fact, for the stationary monodomain Beretta et al. in \cite{BCPcav} show that one boundary measurement is enough to recover a collection of well-separated planar cavities. Moreover, in this case, it is possible to implement an efficient reconstruction method of the cavities based on a phase field approach, \cite{art:BCPR}.

In this paper we extend the results obtained in \cite{BCPcav} to the three-dimensional monodomain time-dependent model.  More precisely, we consider the following boundary value problem
\begin{equation}
\left\{
\begin{aligned}
\partial_t u - div(K_0\nabla u) + f(u,w) &= 0 \qquad &\text{in } \Omega_D\times (0,T), \\
K_0 \nabla u \cdot \normal &= 0 \qquad &\text{on }  \partial\Omega_D\times (0,T), \\
\partial_t w + g(u,w) &= 0 \qquad &\text{in } \Omega_D\times (0,T) , \\
u(\cdot,0) = u_0 \qquad w(\cdot,0) &= w_0 \qquad &\text{in }  \Omega_D,
\end{aligned}
\right.
\label{probcavintro}
\end{equation}
where $\Omega_D=\Omega\backslash D$,  $D\subset\Omega\subset \R^d$, $\nu$ is the outward unit normal vector to the boundary $\partial\Omega_D$, $K_0$ the conductivity tensor. Also, we will analyze the problem considering different types of nonlinearities that are suited to describing cardiac electrophysiology:
the {\it Aliev-Panfilov} model
%\textcolor{red}{Andrea: cosa dobbiamo citare? Non appare la referenza nel comando tex.}. 
\begin{equation}
	%\vspace{-0.25cm}
f(u,w) = Au(u-a)(u-1) + uw \qquad g(u,w) = \epsilon (Au(u-1-a) +  w),\quad A>0,\,0<a<1.
\label{eq:AP}
\end{equation}
the \textit{FitzHugh-Nagumo} model, where
\begin{equation}
\vspace{-0.25cm}
f(u,w) = Au(u-a)(u-1) + w \qquad g(u,w) = \epsilon (\gamma w-u),
\label{eq:FHN}
\end{equation}
and the \textit{Rogers-McCulloch} model, with
\begin{equation}
\vspace{-0.25cm}
f(u,w) = Au(u-a)(u-1) + uw \qquad g(u,w) = \epsilon (\gamma w-u),
\label{eq:RMC}
\end{equation}
see \cite{book:pavarino} for a general overview. 

In our application, the variable $u$ represents the transmembrane potential that propagates through the heart tissue. The heart muscle has a structure consisting of multiple fibers arranged in superimposed sheets, also known as laminas. This arrangement gives rise to preferred directions in the diffusion of the electrical stimulus, which are captured in our model through the use of a tensor-valued conductivity coefficient $K_0$.
The nonlinear reaction terms appearing in the system account for the presence of ionic currents across the cell membrane, modeling the peculiar nonlinear evolution of the voltage, characterized by the propagation of an initial pulse, a plateau phase, and a slow repolarization (see \cite{book:sundes-lines}). Furthermore, $w$ is the so-called gating variable, which represents the number of open channels per unit area of the cellular membrane and thus regulates the transmembrane currents. Finally, $D$ models the infarcted scar and $u_0$ the initial activation of the tissue, arising from the propagation of the electrical impulse in the cardiac conduction system.
Our main result concerns the determination of the cavity $D$ from one measurement of the potential $u$ on $\Sigma\times (0, T)$. For the forward problem, we prove the existence, and global uniqueness of classical solutions, and derive key uniform bounds via the construction of lower and upper solutions. We then use the properties of solutions to derive the uniqueness of the solution to the inverse problem. 
To accomplish this we use the structure of the coupled parabolic system and the uniform boundedness of its solutions showing that the transmembrane potential $u$ is the solution of a linear parabolic integro-differential equation. As a byproduct of the results obtained in \cite{book:escvess} and \cite{art:Vessella2003} a Three Cylinder Inequality is obtained for this class of equations. This guarantees the unique continuation property needed to prove the uniqueness of the solution to the inverse problem.

The plan of the paper is the following: in Section 2 we discuss the well-posedness of the forward problem. In Section 3 we first present the preliminary results needed to prove the main uniqueness result (Theorem 3.1).

\section{Analysis of the direct problem}
 \label{direct}

We consider the initial-boundary value problem
\begin{equation}
\left\{
\begin{aligned}
\partial_t u - div(K_0\nabla u) + f(u,w) &= 0 \qquad &\text{in } \Omega_D\times (0,T), \\
K_0 \nabla u \cdot \normal &= 0 \qquad &\text{on }  \partial\Omega_D\times (0,T), \\
\partial_t w + g(u,w) &= 0 \qquad &\text{in } \Omega_D\times (0,T) , \\
u(\cdot,0) = u_0 \qquad w(\cdot,0) &= w_0 \qquad &\text{in }  \Omega_D,
\end{aligned}
\right.
\label{probcav}
\end{equation}
where $\Omega_D=\Omega\backslash D$,  $D\subset\Omega$, $\nu$ is the outward unit normal vector to the boundary $\partial\Omega_D$, $K_0$ the conductivity tensor and the nonlinearities $f$ and $g$ are of the kind mentioned in the introduction i.e. (\ref{eq:AP}), (\ref{eq:FHN}), or (\ref{eq:RMC}).

%Regarding the non linear terms, a well-established model suited to describing cardiac electrophysiology is specified by the {\it Aliev-Panfilov} nonlinearities
%\cite{book:pavarino}. 
%\textcolor{red}{Andrea: cosa dobbiamo citare? Non appare la referenza nel comando tex.}. 
%\begin{equation}
%	\vspace{-0.25cm}
%f(u,w) = Au(u-a)(u-1) + uw \qquad g(u,w) = \epsilon (Au(u-1-a) +  w),\,\,A>0,\quad 0<a<1.
%\label{eq:AP}
%\end{equation}

%\smallskip
%Some alternatives to the choice of the Aliev-Panfilov model are given for instance 
%by the \textit{FitzHugh-Nagumo} model, where:
%\begin{equation}
%\vspace{-0.25cm}
%f(u,w) = Au(u-a)(u-1) + w \qquad g(u,w) = \epsilon (\gamma w-u),
%\label{eq:FHN}
%\end{equation}

%\smallskip
%and the \textit{Rogers-McCulloch} model, with:
%\begin{equation}
%\vspace{-0.25cm}
%f(u,w) = Au(u-a)(u-1) + uw \qquad g(u,w) = \epsilon (\gamma w-u),
%\label{eq:RMC}
%\end{equation}
%see \cite{book:pavarino} for a general overview. 

In this section, we discuss the well-posedness of problem \eqref{probcav} in the case of the aforementioned nonlinearities $f$ and $g$.
First, let us introduce some preliminary notation, definitions, and our main assumptions.\\
Let $B'_{r_0}$ be the open ball  of radius $r_0$ in $\mathbb{R}^{d-1}$ centered at the origin and denote by $Q_{r_0,2M_0}=B'_{r_0}\times [-2M_0,2M_0]\subset \R^d$.

\begin{definition}
  \label{def:2.1} (${C}^{k+\alpha}$ regularity)
Let $E$ be a domain in ${\R}^{d}$. 
%facciamo tutto con $d$ o tutto con $3$ ?
Given $k$,
$\alpha$, $k\in \N$, $0<\alpha\leq 1$, we say that $ \partial E$
is of \textit{class ${C}^{k+\alpha}$ with
constants $r_{0}$, $M_{0}>0$}, if, for any $P \in \partial E$, there
exists a rigid transformation of coordinates under which we have
$P=0$ and
\begin{equation*}
  E \cap B_{r_{0}}(0)=\{(x',x_d)\in Q_{r_0,2M_0} \quad | \quad
x_{d}>\varphi(x')
  \},
\end{equation*}

%\textbf{Dario: mi sembra che $\subset$ debba sostituire l'uguale e $d$ il $3$}

where $\varphi$ is a ${C}^{k+\alpha}$ function on $B'_{r_{0}}$
satisfying
\begin{equation*}
\varphi(0)=0,
\end{equation*}
\begin{equation*}
\nabla \varphi (0)=0, \quad \hbox{when } k \geq 1,
\end{equation*}
\begin{equation*}
\|\varphi\|_{{C}^{k+\alpha}(B'_{r_{0}}(0))} \leq M_{0}r_{0}.
\end{equation*}

\medskip
\noindent When $k=0$, $\alpha=1$, we also say that $E$ is of
\textit{Lipschitz class with constants $r_{0}$, $M_0$}.
\end{definition}

\begin{assumption}\label{as:1}

\item $\Omega \subset \R^d$ bounded domain, $d=2,3$, and $\partial \Omega \in C^{2+\alpha}$, $0< \alpha < 1$
\end{assumption}
 \begin{assumption}\label{as:2a}
%\begin{itemize}
$K_0(x)$ is a symmetric matrix  satisfying the boundedness and ellipticity conditions 
\begin{equation}
\label{ellipcond}
0<\lambda^{-1}\|\mathbf{\xi}\|^2\leq\xi^T K_0(x)\xi\leq\lambda\|\mathbf{\xi}\|^2, \ \\\ \forall\xi\in\R^d, \forall x\in\Omega
\end{equation}
where $\lambda>1$.
%{\bf \item[(ii)]
%\begin{equation}
%\left\{a_{ij}(x)\right\}_{i,j=1}^2\in L^{\infty}(\Omega).
%\end{equation}
%questa condizione è già contenuta nella precedente e quindi da togliere %(Cristina)}
%\end{itemize}
\end{assumption}
\begin{assumption}\label{as:2b}
$K_0$ is $C^{1+\alpha}(\overline{\Omega})$
%Lipschitz continuous i.e. there exists a positive constant $\Lambda$ such that 
%\[
%|K_0(x)-K_0(y)|\leq \Lambda|x-y|,\quad \forall x,y,\in\overline{\Omega}
%\]
\end{assumption}
\begin{assumption}\label{as:3}	
 
 $D\subset\Omega, \partial D\in C^{2,\alpha}$
  and
	\begin{equation}
	%\vspace{-0.25cm}
	 dist(D,\partial\Omega) \geq d_0 >0.
	%\vspace{-0.25cm}
	%\label{eq:separ}
\end{equation}
\end{assumption}
	%\vspace{-0.25cm}
%	\item $K_0, K_1 \in C^{2}(\overline{\Omega};\R^{d \times d})$ are symmetric matrix-valued functions in $\Omega$;
%	$\forall x \in \overline{\Omega}$, the matrices $K_0(x)$ and $K_1(x)$ admit $d$ positive eigenvalues $k_{0,1} \leq \ldots \leq k_{0,d}$ and $k_{1,1} \leq \ldots \leq k_{1,d}$ respectively, associated to the same eigenvectors $\vec{e}_1(x),\ldots \vec{e}_d(x)$ such that $k_{1,i} \leq k_{0,i}$ $\forall i = 1,\ldots,d$;
\begin{assumption}\label{as:4}
	$u_0 \in C^{2+\alpha}(\overline{\Omega})$ $(0< \alpha < 1)$, $w_0 \in C^{\alpha}(\overline{\Omega})$, $K_0 \nabla u_0 \cdot \normal = 0$ on $\partial\Omega$,
with $|K_0\nu\cdot\nu|>0 $ on $\partial\Omega$. 
 Furthermore, $(u_0,w_0) \in \mathrm{S}$ where 
 $\mathrm{S}:=[0,1+a]\times[0,\frac{A(1+a)^2}{4}]$ for the Aliev-Panfilov and Rogers-McCulloch models and for the FitzHugh-Nagumo model  $S:=[-m,m]\times [-m/\gamma,m/\gamma]$ with $m$ such that 
 \[
    \begin{cases}
    m\ge K_+,\quad \mathrm{if}\quad a\le (\gamma A)^{-1};\\
   m\in (0,K_-]\cup [K_+,+\infty) \quad \mathrm{if}\quad a> (\gamma A)^{-1}  
\end{cases}
%\label{ulfhn}
\]
and $K_{\pm}=\frac{1}{2}\big [ (a+1)\pm \sqrt{(a+1)^2+4((\gamma A)^{-1}-a)}.$
\end{assumption}
We now state the main results regarding the well-posedness of problem \eqref{probcav}.
The proofs of the subsequent results are based on a fixed point argument that leads to a local in-time existence and uniqueness result for classical solutions. Then the properties of the nonlinearities allow the construction of constant upper and lower solutions implying uniform boundness of solutions and global existence in time. This approach is an adaptation of that given in \cite{book:pao} (Chapter 8, Sections 9 and 11). Hence, we will give here a sketch of the proofs, highlighting the differences with respect to the treatment in \cite{book:pao}.

In the following we use the notations: 
\begin{equation}
Q_T:=\Omega_D\times (0,T)\,;\quad\quad S_T:=\partial\Omega_D\times (0,T)\,.
\end{equation}
We assume for the moment that the non linear terms in \eqref{probcav} satisfies a global Lipschitz condition with respect to $(u,w)$, that is
\begin{equation}
\vspace{-0.25cm}
|f(t,v,\w)-f(t,v',\w')|\le M_1\big ( |v-v'|+|\w-\w'|\big )\,,
\nonumber
\end{equation}
\begin{equation}
|g(t,v,\w)-g(t,v',\w')|\le M_2\big ( |v-v'|+|\w-\w'|\big )\,,
\label{lipmod}
\end{equation}
for some positive $M_1$, $M_2$ and for any $(v,\w)\in\R^2$, $t\in\R$.

\begin{theorem}
Let Assumptions \ref{as:1}-\ref{as:4} and \eqref{lipmod} hold. Then, problem \eqref{probcav} admits a unique classical solution $(u,w)$, namely $u\in C^{2+\alpha,1+\alpha/2}(\overline{{Q}_T})$,
$w\in C^{\alpha,1+\alpha/2}(\overline{{Q}_T})$. 
%Moreover, $(u,w)\in \mathrm{S}$ in  $Q_T$. 
%for each $(x,t) \in \overline{Q}_T^D$.
\label{th:1}
\end{theorem}

\begin{proof}
Let us define the new variables $v:=e^{-\kappa t}u$, $\mathrm{w}:=e^{-\kappa t}w$, where $\kappa>0$ is a constant to be chosen. Then, problem \eqref{probcav} becomes
\begin{equation}
\left\{
\begin{aligned}
\partial_t v - div(K_0\nabla v)+\kappa v  &= f^*(t,v,\w) \qquad &\text{in } Q_T, \\
K_0 \nabla v \cdot \normal &= 0 \qquad &\text{on }  S_T, \\
\partial_t \w + \kappa\w&= g^*(t,v,\w) \qquad &\text{in } Q_T , \\
v(\cdot,0) = u_0\,, \qquad \w(\cdot,0) &= w_0 \qquad &\text{in }  \Omega_D,
\end{aligned}
\right.
\label{probcavmod}
\end{equation}
where 
\begin{equation}
\vspace{-0.25cm}
f^*(t,v,\w) = -e^{-\kappa t}\,f(e^{\kappa t}v,e^{\kappa t}\w)\qquad 
g^*(t,v,\w) = -e^{-\kappa t}\,g(e^{\kappa t}v,e^{\kappa t}\w)\,.
\label{fgmod}
\end{equation}
\smallskip
Note that the same global Lipschitz condition \eqref{lipmod} holds for $f^*$ and $g^*$
with constants independent of $t$; hence, in the following we will drop the explicit dependence on $t$ of $f^*$ and $g^*$.
We will reformulate the problem \eqref{probcavmod} as a fixed point equation in the Banach space of the pairs $(v,\w)$, $v$, $\w$ $\in C(\overline{Q_T})$, provided with the sup norm.
Let us first define the operator
\begin{equation}
A\,:\, \mathcal{D}(A)\rightarrow C({Q_T})\,;\qquad
A v:= \partial_t v - div(K_0\nabla v)+\kappa v\,,
\label{defA}
\end{equation}
where 
%\begin{equation}
%\mathcal{D}(A):=
%\Big\{v\in \mathcal{C}(\overline{Q_T})\cap \mathcal{C}^{2,1}({Q_T})\,,
%\,K_0\nabla v\cdot\nu\big |_{S_T}=0\,, v(\cdot,0)=u_0\,\Big\}
%\end{equation}

\begin{equation}
\mathcal{D}(A):=
\Big\{v\in C^{2+\alpha,1+\alpha/2}(\overline{Q_T}) \,,
\,K_0\nabla v\cdot\nu\big |_{S_T}=0\,, v(\cdot,0)=u_0\,\Big\}.
\label{defdomA}
\end{equation}

We further observe that the problem for $\w$ in \eqref{probcavmod} is equivalent to the integral equation
\begin{equation}
\w(x,t)=e^{-\kappa t} w_0(x)+\int_0^t e^{-\kappa (t-\tau)}g^*\big (v(x,\tau),\w(x,\tau)\big )d\tau\,.
\label{inteqw}
\end{equation}
Hence, by denoting with $G(v,\w)$ the right hand side of \eqref{inteqw}, we can write \eqref{probcavmod} in the operator form
\begin{equation}
\left\{
\begin{aligned}
A v  &= f^*(v,\w)\,, \\
\w &= G(v,\w)  \,. \\
\end{aligned}
\right.
\label{probcavope}
\end{equation}

Now, in order to get the fixed point equation we discuss the invertibility of the operator $A$  following \cite{book:pao}. To this aim, by exploiting the regularity of $K_0$, we write the operator $A$ in non-divergence form.  
%the $\mathcal C^{1+\alpha}$ regularity of the tensor valued conductivity $K_0$ to apply  
%\textcolor{red}{(EE quanto regolari?)} \textcolor{blue}{$\mathcal C^{1,\alpha}$ ?} of the coefficients of $A$ 
%\emph{written in non divergence form}. 
Then by known results in the theory of linear parabolic equations, see for example Theorem 5.1.17 in \cite{lunardi},  we can ensure the existence of a unique solution $v\in \mathcal{D}(A)$ such that $A v=h$ provided $h\in C^{\alpha,\alpha/2}(\overline{Q}_T)$.
%has a unique (classical) solution in $C^{2+\alpha,1+\alpha/2($ provided $h\in C^{\alpha,\alpha/2(\overline{Q_T})$ 
%and  $v$ satisfies $$ the homogeneous boundary condition and an initial condition  satisfying the assumptions  
%satisfied in the pointwise sense, see e.g. \cite{book:pao}, chapter 2, theorem 1.2). 

Then the operator $A^{-1}$ is well defined on the subspace 
$X=C ^{\alpha,\alpha/2}(\overline{Q_T})$; 
%$X=C(\overline{Q_T})\cap C ^{\alpha,\alpha/2}(Q_T)$
furthermore, it is easily verified that  the operators $f^*$ and $G$ map $X\times X$ into $X$. Hence we finally get

\begin{equation}
\left\{
\begin{aligned}
v  &= A^{-1}f^*(v,\w)\,, \\
\w &= G(v,\w)  \,. \\
\end{aligned}
\right.
\label{probcavfix}
\end{equation}
with $(v,\w)\in X\times X$.

By denoting with $|\cdot |_0$ the sup norm in $C(\overline{Q_T})$,
we can now follow the same pattern as in the proof of Lemma 9.1 in \cite{book:pao} Chapter 8, Section 9, using ellipticity of $A$ and the properties of the functions in $\mathcal{D}(A)$ at the extremum points to get the estimate
\begin{equation}
|Av-Av'|_0\ge \kappa |v-v'|_0\,\qquad v\,,v'\,, \in \mathcal{D}(A)\,.
\label{eq:estA}
\end{equation} 

\smallskip
The only delicate point is that in \cite{book:pao} it is assumed a Neumann (or Dirichlet) boundary condition on $S_T$, which is different from the condition in \eqref{defdomA}. Nevertheless, it can shown that the proof still works by the restriction on the conormal derivative in assumption \ref{as:4}.

By \eqref{probcavfix}, \eqref{eq:estA}, we now have
\begin{equation}
|A^{-1}f^*(v,\w)-A^{-1}f^*(v',\w')|_0\le \frac{1}{\kappa} |f^*(v,\w)-f^*(v',\w')|_0
\le \frac{M_1}{\kappa}(|v-v'|_0+|\w-\w'|_0)\,.
\label{eq:estinvAf}
\end{equation}
Finally, by the definition of $G$ (see \eqref{inteqw}) it follows readily
\begin{equation}
\begin{aligned}
|G(v,\w)-G(v',\w')|_0\le \int_0^t e^{-\kappa (t-\tau)}
\big |g^*\big (v(x,\tau),\w(x,\tau)-g^*\big (v'(x,\tau),\w'(x,\tau)\big |\,d\tau \\
\le \frac{M_2}{\kappa}(|v-v'|_0+|\w-\w'|_0)\,.
\end{aligned}
\label{eq:estG}
\end{equation}

By choosing $\kappa>\max\{M_1,M_2\}$ the operator defined by the right hand side of \eqref{probcavfix} is a contraction in $X\times X$ with respect to the sup norm, so that there is a unique fixed point. Then, by exploiting known regularities results for parabolic boundary value problems (see e.g. \cite{lunardi}, section $5.1.2$) it can be shown that the solution to problem \eqref{probcavmod}, as well as to \eqref{probcav},
has the required regularity properties. 
\end{proof}

\begin{remark}
By the previous theorem, the solution $(u,w)$ to problem \eqref{probcav} is the limit of the sequence $(u^{(k)},w^{(k)})$ defined by the iteration process
\begin{equation}
\left\{
\begin{aligned}
\partial_t u^{(k)} - div(K_0\nabla u^{(k)})  &= -f(u^{(k-1)},w^{(k-1)}) \qquad &\text{in } Q_T \\
K_0 \nabla u^{(k)} \cdot \normal &= 0 \qquad &\text{on }  S_T, \\
\partial_t w^{(k)} &= - g(u^{(k-1)},w^{(k-1)})  \qquad &\text{in } Q_T , \\
u^{(k)}(\cdot,0) = u_0 \qquad w^{(k)}(\cdot,0) &= w_0 \qquad &\text{in }  \Omega_D,
\end{aligned}
\right.
\label{probcaviter}
\end{equation}
with $k=1,2,...$ and where $(u^{(0)},w^{(0)})$ is any pair in 
$C^{2+\alpha}(\overline{Q_T})\times C^{\alpha}(\overline{Q_T})$.
  \end{remark}

We now need to replace the global Lipschitz condition on the non linear terms with a local one in order to include the cases of the non linearities \eqref{eq:AP}, \eqref{eq:FHN} and \eqref{eq:RMC}.

Preliminarly, we introduce further definitions and results from \cite{book:pao}, Chapter 8, Section 9.
\begin{definition}
  \label{def:gculs}
Two pairs of functions $(\overline u, \overline w)$, $(\underline u,\underline w)$, each one in 
$\mathcal{C}^{2+\alpha,1+\alpha/2}(\overline{Q_T)}\times \mathcal{C}^{\alpha,1+\alpha/2}(Q_T)$, are called generalized coupled upper and lower solutions of \eqref{probcav} if $\overline u\ge \underline u$, $\overline w\ge \underline w$ and if they satisfy
\medskip
\begin{equation}
\left\{
\begin{aligned}
\partial_t \overline u - div(K_0\nabla \overline u) + f(\overline u,z) &\ge 0 \qquad &\text{in } Q_T, 
\quad &\forall z\in [\underline w,\overline w],\\
K_0 \nabla \overline u \cdot \normal &\ge 0 \qquad &\text{on }  S_T, \\
\partial_t \overline w + g(v,\overline w) &\ge 0 \qquad   &\text{in } Q_T, 
\quad &\forall  v\in [\underline u,\overline u],\\
\overline u(\cdot,0) \ge u_0 \qquad \overline w(\cdot,0) &\ge w_0 \qquad &\text{in }  \Omega_D,
\end{aligned}
\right.
\label{upperprob}
\end{equation}
\medskip
\begin{equation}
\left\{
\begin{aligned}
\partial_t \underline u - div(K_0\nabla \underline u) + f(\underline u,z) &\le 0 \qquad &\text{in } Q_T, 
\quad &\forall z\in [\underline w,\overline w],\\
K_0 \nabla \underline u \cdot \normal &\le 0 \qquad &\text{on }  S_T, \\
\partial_t \underline w + g(v,\underline w) &\le 0 \qquad   &\text{in } Q_T, 
\quad &\forall  v\in [\underline u,\overline u],\\
\underline u(\cdot,0) \le u_0 \qquad \underline w(\cdot,0) &\le w_0 \qquad &\text{in }  \Omega_D,
\end{aligned}
\right.
\label{lowerprob}
\end{equation}
\end{definition}

\medskip
We now suppose that for every $R>0$ there are constants $M_1=M_1(R)$, $M_2=M_2(R)$ such that
\begin{equation}
\vspace{-0.25cm}
|f(v,\w)-f(v',\w')|\le M_1\big ( |v-v'|+|\w-\w'|\big )\,,
\nonumber
\end{equation}
\begin{equation}
|g(v,\w)-g(v',\w')|\le M_2\big ( |v-v'|+|\w-\w'|\big )\,,
\label{liploc}
\end{equation}
whenever $|v|+|w|\le R$, $|v'|+|w'|\le R$. 

\medskip
Then we can prove (see \cite{book:pao}, chapter 8, theorem 9.3)

\begin{theorem}
\label{theobounds}
Let  $(\overline u,\overline w)$, $(\underline u,\underline w)$ be generalized coupled upper and lower solutions to \eqref{probcav}, where $f$, $g$ satisfy \eqref{liploc}. Then the problem \eqref{probcav} has a unique solution $(u,w)$ with $\underline u\le u\le \overline u$, $\underline w\le w\le\overline w$.  
\end{theorem}

\begin{proof}
Let us define
\begin{equation}
 \mathcal{S}:=\Big \{(v,z)\in \mathcal{C}^{2+\alpha,1+\alpha/2}(\overline{Q_T)}\times \mathcal{C}^{\alpha,1+\alpha/2}(Q_T)\,:\,\underline u\le v\le \overline u,\quad \underline w\le z\le\overline w\Big\}
\end{equation}

For any pair $(v,z)\in \mathcal{S}$ let $(u^*, w^*)$ be the solution of the \emph{linear} problem
\begin{equation}
\left\{
\begin{aligned}
\partial_t  u^* - div(K_0\nabla  u^*) + k_1 u^* &=M_1 v -f(v,z)\qquad &\text{in } Q_T \\
K_0 \nabla  u^* \cdot \normal &= 0 \qquad &\text{on }  S_T, \\
\partial_t  w^* + M_2 w^* &= M_2 z-g(v,z) \qquad &\text{in } Q_T , \\
u^*(\cdot,0) = u_0 \qquad w^*(\cdot,0) &= w_0 \qquad &\text{in }  \Omega_D,
\end{aligned}
\right.
\label{linprob}
\end{equation}
where $M_1$, $M_2$ are the Lipschitz constants in \eqref{liploc} and we assume $R$ larger than 
any $|v|+|z|$ for $(v,z)\in \mathcal{S}$. Let us point out that this value of $R$ only depends on $(\overline u,\overline w)$ and $(\underline u,\underline w)$. Let $U:=\overline u- u^*$; by \eqref{upperprob} and \eqref{linprob} we get
\begin{equation}
 \partial_t U-div(K_0\nabla U)+M_1 U\ge -f(\overline u,z)+M_1(\overline u-v)+f(v,z),
\end{equation}
Now, by the Lipschitz condition and recalling that $v\le \overline u$, we get
$$f(v,z)-f(\overline  u,z)\ge -M_1(\overline  u-v).$$

Hence,
\begin{equation}
 \partial_t U-div(K_0\nabla U)+M_1 U\ge 0, 
\end{equation}
together with the boundary condition $K_0 \nabla U \cdot \normal \ge 0$ and the initial condition
$U(x,0)\ge 0$.

Thus, by the maximum principle for parabolic operators $U\ge 0$ in $\overline{Q_T}$, which yields
$u^*\le \overline  u$. An analogous argument with the lower solution gives $u^*\ge \underline u$.

Similarly, by defining $W=\overline  w- w^*$ we readily find that
\begin{equation}
 \partial_t W+M_2 W\ge 0, 
\end{equation}
together with the initial condition $W(x,0)\ge 0$, so that we still have $W\ge 0$ and $w^*\le \overline  w$. Finally, we get $w^*\ge \underline w$ by considering the lower solution.

Now, by suitably modifying the functions $f$, $g$ outside $[\underline u, \overline u ]
\times [\underline w, \overline{w} ]$, we can still assume that they are globally Lipschitz; then, 
for any $(u^{(0)},w^{(0)})\in \mathcal{S}$, consider the sequence defined by
\begin{equation}
\left\{
\begin{aligned}
\partial_t u^{(k)} - div(K_0\nabla u^{(k)}) + M_1 u^{(k)}  &=M_1 u^{(k-1)}  -f(u^{(k-1)},w^{(k-1)}) \qquad &\text{in } Q_T \\
K_0 \nabla u^{(k)} \cdot \normal &= 0 \qquad &\text{on }  S_T, \\
\partial_t w^{(k)} + M_2 w^{(k)}&=M_2 w^{(k-1)} -g(u^{(k-1)},w^{(k-1)})  \qquad &\text{in } Q_T , \\
u^{(k)}(\cdot,0) = u_0 \qquad w^{(k)}(\cdot,0) &= w_0 \qquad &\text{in }  \Omega_D,
\end{aligned}
\right.
\label{probitermod}
\end{equation}
with $k=1,2,...$ . By the previous bounds, every pair $(u^{(k)}, w^{(k)})$ belongs to $\mathcal{S}$. Hence, we can still apply the arguments of theorem \ref{th:1} with the global Lipschitz condition to conclude (as for the sequence \eqref{probcaviter}) that the sequence $(u^{(k)}, w^{(k)})$ converges uniformly to a unique solution $(u,w)$ of the problem
\ref{probcav} and that $\underline u\le u\le \overline u$, $\underline w\le w\le\overline w$.
\end{proof}

In order to apply theorem \ref{theobounds} to our models we need to find suitable upper and lower solutions for $f$ and $g$ as in \eqref{eq:AP}, \eqref{eq:FHN}, \eqref{eq:RMC}. Actually, it is convenient to look for \emph{constants} upper and lower solutions:
\begin{lemma}
For each one of the nonlinear terms \eqref{eq:AP}, \eqref{eq:FHN}, \eqref{eq:RMC} there are constants upper and lower solutions.\\
Precisely, in the Aliev-Panfilov and Rogers-McCulloch models one can take
$\underline u=0$, $\underline w=0$, while 
\begin{equation}
    \overline{u}=1+a\,,\quad\quad \overline{w}\ge A(1+a)^2/4
    \label{APbds}
\end{equation}
in the Aliev -Panfilov model and 
\begin{equation}
   \overline{u}\in (0,a)\cup(1,+\infty)\,,\quad\quad \overline{w}\ge \overline{u}/\gamma,
    \label{RMCbds}
\end{equation}
in the Rogers-McCulloch model.\\ 
Finally, in the Fitzhugh-Nagumo model one can take $\overline{u}=m=-\underline{u}$,
$\overline{w}=m/\gamma=-\underline{w}$, with 
\begin{equation}
    \begin{cases}
    m\ge K_+,\quad \mathrm{if}\quad a\le (\gamma A)^{-1};\\
   m\in (0,K_-]\cup [K_+,+\infty) \quad \mathrm{if}\quad a> (\gamma A)^{-1}  
\end{cases}
\label{ulfhn}
\end{equation}
and $K_{\pm}=\frac{1}{2}\big [ (a+1)\pm \sqrt{(a+1)^2+4((\gamma A)^{-1}-a)}$.
\end{lemma}
\begin{proof}
We have to check definition \ref{def:gculs} for the three nonlinearities.\\
If $f(u,w)=Au(u-a)(u-1)+uw$, $g(u,w)=\epsilon(Au(u-1-a)+w)$, we have $f(0,w)=0$, $g(u,0)= \epsilon Au(u-1-a) \le 0$ for any $u\in [0,1+a]$; moreover
$f(1+a,w)=(1+a)(Aa+w)\ge 0$. Finally, $g(u,\overline w)\ge 0$ if $\overline w\ge Au(1+a-u)$
and by an elementary calculation we get
\begin{equation}
\max_{0\le u\le 1+a} Au(1+a-u)=A\frac{(1+a)^2}{4}.
\end{equation}
Hence, in the Aliev-Panfilov model theorem \ref{theobounds} applies if the values $(u_0,w_0)$ of the initial data lie in the rectangle $\mathrm{S}=[0,1+a]\times[0,\frac{A(1+a)^2}{4}]$.

Let us now consider the Rogers-McCulloch model, where $f$ is unchanged and $g(u,w)=\epsilon(\gamma w-u)$. Clearly, we still have $f(0,w)=0$ for every $w$ and $g(u,0)=-\epsilon u\le 0$ for any $u\ge 0$. Now, since

$$f(\overline u,w)=A \overline{u}\big ( (\overline{u}-a)(\overline{u}-1)+w\big )\,,$$
we see that the condition $f(\overline u,w)\ge 0$ for any $w\ge 0$ holds provided that
$(\overline u-a)(\overline u-1)\ge 0$, i.e. $\overline u\in (0,a)\cup (1+\infty)$.
Finally, by choosing $\overline{w}\ge \overline u/\gamma$ we get
$$g(u,\overline{w})=\epsilon(\gamma\overline{w}-u)\ge 0\,\quad\quad \forall\,\,u\in [0,\overline u]\,.$$
Thus, in this model we can take $S=[0,\overline u]\times [0,\overline w]$, with
the above $\overline u$, $\overline w$.

We are left to consider the Fitzhugh-Nagumo model where
\begin{equation}
\vspace{-0.25cm}
f(u,w) = Au(u-a)(u-1) + w\,, \qquad g(u,w) = \epsilon (\gamma w-u)
\,.
\end{equation}
As can be readily checked, in this case we can choose $S=[-m,m]\times [-m/\gamma,m/\gamma]$, provided that $m$ satisfies 
$$-A(m+a)(m+1)+\frac{1}{\gamma}\le 0\le A(m-a)(m-1)-\frac{1}{\gamma}\,.$$
Hence, by elementary calculations we deduce the conditions \eqref{ulfhn} on $m$.
\end{proof}
\begin{remark}
    It is worthwhile to remark that non-negative constant lower and upper solutions also exist in the Fitzhugh-Nagumo model if the product $\gamma A$ is large enough.\\
    In fact, we first observe that by taking
    $\underline u>0$ and $\underline w=0$, we get $g(u,0)=-\epsilon u < 0$; moreover, the inequality
    $$f(\underline u,w)=A \underline u(\underline u-a)(\underline u-1)+w\le 0$$
could be satisfied if $a<\underline u<1$ and $w>0$ is small enough.
Actually, by defining $m_a=-\min_{a<u<1}u(u-a)(u-1)>0$ and by choosing $\underline u$ at the minimum point, the above bound is satisfied if 
$\overline w\le A m_a$. \\
Finally, the upper solutions must satisfy
$$g(u,\overline w)= 
\epsilon (\gamma \overline w-u)\ge  \epsilon (\gamma \overline w- \overline u)\ge 0$$ and 
$$f(\overline u,w)=A \overline u(\overline u-a)(\overline u-1)+w\ge 0.$$
These two conditions lead to $\overline u\le \gamma\overline w\le \gamma A m_a$
and to $\overline u>1$ respectively. The last inequalities are compatible only if $\gamma A>1/m_a$.
\end{remark}
Collecting the results obtained so far, we can state
\begin{theorem}
Let Assumptions \ref{as:1}-\ref{as:4} hold and let $(f,g)$, be defined as in \eqref{eq:AP}, \eqref{eq:FHN}, \eqref{eq:RMC} respectively. We also assume the following initial data:
\begin{itemize}
    \item $(u_0,v_0)\in [0,\overline{u}]\times [0,\overline{w}]$, where $\overline{u}$, $\overline{w}$ satisfy \eqref{APbds} for $(f,g)$ in \eqref{eq:AP} (Aliev-Panfilov model), and \eqref{RMCbds} for $(f,g)$ in \eqref{eq:RMC} (Rogers-McCulloch model);
    \item $(u_0,v_0)\in [-m,m]\times [-m/\gamma,m/\gamma]$, with $m$ 
satisfying \eqref{ulfhn} for $(f,g)$ in \eqref{eq:FHN} (Fitzhugh-Nagumo model).
\end{itemize}
Then, problem \eqref{probcav} admits a unique classical solution $(u,w)$, namely $u\in C^{2+\alpha,1+\alpha/2}(\overline{{Q}_T})$,
$w\in C^{\alpha,1+\alpha/2}(\overline{{Q}_T})$. Moreover, $(u,v)\in \mathcal{S}$ for every $T$, where $\mathcal{S}$ denotes any of the above rectangles depending on the model considered. 
\label{th:final}
\end{theorem}

%\textcolor{red}{Elena: forse e' meglio non dire niente sugli altri modelli perche' su Fitzhugh-Nagumo non sappiamo risolvere il problema inverso }
%A detailed proof is reported in \cite[Theorem 6.1]{phd:Luca}.

\par
%Regarding the perturbed problem \eqref{eq:perturbed}, we note that, although the conductivity tensor and the nonlinear term are discontinuous, we can extend the results obtained in \cite{art:BCP} thanks to the uniform ellipticity to the boundedness of the conductivity tensors and to the form of the reaction term, deriving the following existence and uniqueness result.

\section{Analysis of the inverse problem: Uniqueness}

Consider the initial-boundary value problem
\begin{equation}
\left\{
\begin{aligned}
\partial_t u - div(K_0\nabla u) + f(u,w) &= 0 \qquad &\text{in } \Omega_D, \\
K_0 \nabla u \cdot \normal &= 0 \qquad &\text{on }  \partial\Omega_D, \\
\partial_t w + g(u,w) &= 0 \qquad &\text{in } \Omega_D\times (0,T) , \\
u(\cdot,0) = u_0 \qquad w(\cdot,0) &= w_0 \qquad &\text{in }  \Omega_D,
\end{aligned}
\right.
\label{probcav1}
\end{equation}
In this section, we  will investigate the uniqueness of the inverse problem i.e.
\begin{problem}
\label{inverseproblem}
\label{ext} Assume it is possible to measure the solution to (\ref{probcav1}), $u_{\Sigma\times(0,T)}$ where $\Sigma$ is an open connected portion of $\partial\Omega$. Is it possible to uniquely determine $D$? 
\end{problem}
The answer is positive under the further assumption 
%on the initial datum $u_0$ that its support is close
that the support of the initial datum $u_0$ lies so close to the boundary of $\Omega$ that it never intersects the cavity $D$. More precisely we assume that
\begin{assumption}\label{as:5}
\begin{equation}
\textrm{dist}(\textrm{supp}(u_0),\partial\Omega)\leq d_0/2.
\end{equation}
\end{assumption}
Then the following result holds:
\begin{theorem}\label{uniqueness}
Let Assumptions  \ref{as:1}-\ref{as:4} and \ref{as:5} hold and let $(u_1,w_1)$ and $(u_2,w_2)$  be solutions of \ref{probcav1} corresponding to $D=D_1$ and $D_2$ respectively. If $u_1=u_2$ on $\Sigma\times (0,T)$ where $\Sigma$ is an open portion of $\partial\Omega$ then $D_1=D_2$.
\end{theorem}
To prove the theorem we will state and prove some preliminary results. 
\subsection{Estimates of unique continuation}
We will show that solutions of system (\ref{probcav1}) enjoy the unique continuation property by deriving a Three Cylinder Inequality.
%We will prove that under suitable assumptions on the nonlinearities $f$ and $g$ the answer is yes.\\
%Let us list first of all our main assumptions.
%\subsection{Three cylinder inequality}
%Notation:
We denote by $B_r$ the open ball of radius r in $\R^d$ centered at the origin and by $Q^T_r$ the cylinder $Q^T_r=B_r\times (0,T)$ in $\R^{d+1}$.
\begin{theorem}\label{threecylinderest}
    Let $u\in H^{2,1}\left(Q^T_{R}\right)$ such that
    \begin{equation}\label{disug}
        \left|\partial_t u - div(K_0\nabla u)\right|\leq \Lambda_0\left(|u|+|\nabla u|+\Lambda_1\int_0^t|u(x,s)|ds\right)
\quad\mbox{ in } Q^T_{R},    \end{equation}
where $K_0$ satifies assumptions \ref{as:2a}, \ref{as:2b}, and 
\[u(x,0)=0.\]
There exists a constant $C$ depending on $\lambda$, $\|K_0\|_{C^{1,\alpha}}$ and on $\Lambda_0$ and $\Lambda_1$ such that, for 
\[r\leq\rho\leq R/C\] and $\delta\in(0,T)$, we have
\begin{equation}\label{TC}
    \|u\|_{L^2\left(Q^{T-\delta}_\rho\right)}\leq \left(\frac{CTR}{\delta\rho}\right)^\beta \|u\|^\theta_{L^2\left(Q^{T}_r\right)}\|u\|^{1-\theta}_{L^2\left(Q^{T}_R\right)}
\end{equation}
where 
\begin{equation}
\theta=\frac{\log\frac{R}{C\rho}}{C\log \frac{R}{r}} \mbox{ and } \beta=C\left(\frac{R^2}{T}+\frac{T}{\delta}\right)^C. 
\end{equation}
\end{theorem}
\begin{proof}
    The proof can be obtained by following the proofs in \cite{book:escvess} (that contains the Carleman estimate needed for this result) and \cite{art:Vessella2003} (that contains the proof of how Three Cylinder Inequality follows from Carleman estimate). The main difference in the present case is the integral term in the right-hand side of \eqref{disug}. %This term gives rise to an extra term in the estimate of $\tilde{\zeta}\mathcal{L}\tilde{u}$ in \cite{art:Vessella2003} of the form $\Lambda_0\Lambda_1\int_0^t\left|\tilde{\zeta}\tilde{u}\right|$ that can be easily absorbed by the other terms of the estimate. 
    This term gives rise in the estimates of \cite{art:Vessella2003} to an extra term that can be easily absorbed by the other terms.
    Since the changes that one needs are quite standard, for sake of shortness, we decided not to include further details. 
%In order to complete the proof, a Caccioppoli inequality for solutions to \eqref{disug}, is also needed. We prove it in the appendix.
\end{proof}
\subsection{An auxiliary lemma}
\begin{lemma}\label{energymethod}
Let $\Omega^*\subset \mathbb{R}^3$ be a bounded measurable set such that $\mathcal{H}^2(\partial\Omega^*)<+\infty$ (where $\mathcal{H}^2$ is the two-dimensional Hausdorff measure)  and let  $u \in C^{2+\alpha,1+\alpha/2}(\overline{\Omega}^*\times [0,t^*])$ be a solution of 
\begin{equation}\label{linearprob}
\left\{
\begin{aligned}
\partial_t u - \textrm{div}(K_0\nabla u) + f_1(u)+k_1\int_0^te^{-c_1(t-s)}g_1(u)ds &= 0 \qquad &\text{in } \Omega^*\times(0,t^*), \\
K_0 \nabla u \cdot \normal &= 0 \qquad &\text{on }  \partial\Omega^*\times(0,t^*), \\
u(\cdot,0) &=0 \qquad &\text{in } \Omega^*,
\end{aligned}
\right.
\end{equation}
with 
\[
f_1(u)=a_1(x,t)u
\]
%with $f_1\in C^1(\mathbb{R})$, $f_1(0)=0$ and satisfying
%$$\forall M>0\,\, \exists\, C_0 \textrm{ such that } |f_1'(u)|+|f_1''(u)|\leq C_0, \textrm{ if } |u|\leq M,$$
%$$f_1'(u)\geq -\Lambda_{f_1},\,\forall u\in\mathbb{R}$$ for some positive constant 
%$\Lambda_{f_1}$. 
and there exists a positive constant $C_0$ such that 
\[
|a_1(x,t)|,\,|k_1(x,t)|\leq C_0\,\, \textrm{ for all }(x,t)\in\overline{\Omega}^*\times [0,t^*].
\] 
$c_1$ positive constant, $g_1\in C^1(\mathbb{R})$, $g_1(0)=0$ and  $$\exists\, C_1 \textrm{ such that } |g_1'(u)|\leq C_1. $$
%\textrm{ in } \Omega^*\times(0,t^*) $$
%Also, the Neumann condition has to be understood in the usual weak sense.
Then $u=0$ in $\overline{\Omega}^*\times[0,t^*]$.
\end{lemma}

\begin{proof}
Let us multiply the equation 
\[
\partial_t u - \textrm{div}(K_0\nabla u) + a_1 u+k_1\int_0^te^{-c_1(t-s)}g_1(u)ds = 0 
%\partial_t u - \textrm{div}(K_0\nabla u) + F(u) = 0 
\]
by $e^{-4 C_0 t}u$ and let us integrate over $\Omega^*\times(0,t)$ for any $t\in (0,t^*]$. We then get 
\begin{equation*}
\begin{aligned}
\iint_{\Omega^*\times(0,t)}\bigg\{\partial_s ue^{-4C_0 s}u &- \textrm{div}(K_0\nabla u)e^{-4C_0 s}u + a_1u^2e^{-4C_0 s}\\
&+k_1ue^{-4C_0 s}\int_0^se^{-c_1(s-\tau)}g_1(u)d\tau\bigg\}\,dxds = 0 
\end{aligned}
\end{equation*}
Since the solution $u$ to (\ref{linearprob})  is in $ C^{2+\alpha,1+\alpha/2}(\overline{\Omega}^*\times [0,t^*])$ we can apply Green's formula on sets of finite perimeter (cf. for example \cite{EvansGariepy}) to derive the following identity 
\begin{equation}\label{identity2}
\iint_{\Omega^*\times(0,t)}\textrm{div}(K_0\nabla u)e^{-4C_0 s}u\,dxds=-\iint_{\Omega^*\times(0,t)}K_0\nabla u\cdot \nabla u e^{-4C_0 s}\,dxds
\end{equation}
where we have used the boundary condition.
Note that 
\begin{equation*}
\begin{aligned}
&\iint_{\Omega^*\times(0,t)}\partial_s ue^{-4C_0 s}u\,dxds\\
&=\frac{1}{2} \iint_{\Omega^*\times(0,t)}\frac{\partial}{\partial s}(e^{-4C_0 s}u^2)\,dxds+2C_0 \iint_{\Omega^*\times(0,t)}e^{-4C_0 s}u^2\,dxds\\ 
&=\frac{1}{2}\int_{\Omega^*}e^{-4C_0 t}u^2(x,t)\,dx+
2C_0\iint_{\Omega^*\times(0,t)}e^{-4C_0 s}u^2\,dxds
\end{aligned}
\end{equation*}
where we have used the condition $u(x,0)=0$. Hence, we can write 

\begin{eqnarray*}
\frac{1}{2}\int_{\Omega^*}e^{-4C_0 t}u^2(x,t)\,dx&+&
\iint_{\Omega^*\times(0,t)}K_0\nabla u\cdot \nabla u e^{-4 C_0 s}\,dxds\\
&+&\iint_{\Omega^*\times(0,t)}e^{-4 C_0s}(a_1u^2+2C_0u^2)\,dxds\\
&+&\iint_{\Omega^*\times(0,t)}e^{-4C_0s}k_1u\mathcal{K}[u](s)\,dxds=0
%&+&\iint_{\Omega^*\times(0,t)}e^{-4 \Lambda_{f_1}s}(f_1(u)u+2\Lambda_{f_1}u^2)\,dxds+\iint_{\Omega^*\times(0,t)}\\
%&=&-\iint_{\Omega^*\times(0,t^*)}(a+\frac{\mu}{2})e^{-\mu t}u^2(x,t)\,dxdt
\end{eqnarray*}
where we have set 
$$
\mathcal{K}[u](s)=\int_0^se^{-c_1(s-\tau)}g_1(u)d\tau.
$$
By the condition $|a_1|\leq C_0$ we have that 
%$f_1(u)\geq -\Lambda_{f_1}u$ which gives 
$$
a_1u^2+2C_0 u^2\geq C_0 u^2.
$$
Hence, we can write 
%\begin{eqnarray}\label{identity4}
\begin{equation}
\begin{aligned}
\frac{1}{2}e^{-4C_0 t}\|u(t)\|^2_{L^2(\Omega^*)}+
\lambda^{-1}\iint_{\Omega^*\times(0,t)}|\nabla u|^2e^{-4 C_0 s}\,dxds
+C_0\iint_{\Omega^*\times(0,t)}e^{-4 C_0s}u^2\,dxds\\
\leq -\iint_{\Omega^*\times(0,t)}e^{-4C_0s}k_1u\mathcal{K}[u](s)\,dxds
%\end{eqnarray}
\end{aligned}
\label{identity4}
\end{equation}
where we have also used the coercivity of the matrix $K_0$. Finally, let us bound 
\begin{eqnarray}\label{est1}
\left|\iint_{\Omega^*\times(0,t)}e^{-4 C_0s}k_1u\mathcal{K}[u](s)\,dxds\right|
\leq C_0 \int_0^t e^{-4 C_0 s}\int_{\Omega^*}|u(x,s)||\mathcal{K}[u](s)|dxds
\end{eqnarray}
Observe now that
\begin{eqnarray*}
  \int_{\Omega^*}|u(x,s)||\mathcal{K}[u](s)|dx&\leq&   \left(\int_{\Omega^*}|u(x,s)|^2dx\right)^{1/2} \left(\int_{\Omega^*}|\mathcal{K}[u](s)|^2dx\right)^{1/2}\\
  &=& \|u(s)\|_{L^2(\Omega^*)}\left(\int_{\Omega^*}\left(\int_0^s e^{-c_1(s-\tau)}g_1(u)d\tau\right)^2 dx\right)^{1/2}\\
  &\leq& C_1  \|u(s)\|_{L^2(\Omega^*)}\left(\int_{\Omega^*}\left(\int_0^s e^{-c_1(s-\tau)}|u|d\tau\right)^2 dx\right)^{1/2}\\
  &\leq & C_1 \|u(s)\|_{L^2(\Omega^*)}\int_0^s\|u(\tau)\|_{L^2(\Omega^*)}d\tau\\
  &\leq& C_1t^* \|u(s)\|_{L^2(\Omega^*)}\max_{0\leq \tau\leq s}\|u(\tau)\|_{L^2(\Omega^*)}
\end{eqnarray*}
where we have used Minkowski inequality, that is
\begin{equation*}
\left(\int_{\Omega^*}\left(\int_0^s|u(x,\tau)| d\tau\right)^2 dx\right)^{1/2}\leq \int_0^s\left(\int_{\Omega^*}|u(x,\tau)|^2dx\right)^{1/2}d\tau
\end{equation*}
and plugging it in (\ref{est1}) and in (\ref{identity4}) we derive the following estimate 
\begin{equation*}
\begin{aligned}
\frac{1}{2}e^{-4C_0 t}\|u(t)\|^2_{L^2(\Omega^*)}+
C_0\iint_{\Omega^*\times(0,t)}|\nabla u|^2e^{-4 C_0s}\,dxds
+C_0\iint_{\Omega^*\times(0,t)}e^{-4 C_0s}u^2\,dxds\\
\leq C\int_0^t \|u(s)\|_{L^2(\Omega^*)}\max_{0\leq \tau\leq s}\|u(\tau)\|_{L^2(\Omega^*)}ds
\leq C \int_0^t \max_{0\leq \tau\leq s}\|u(\tau)\|^2_{L^2(\Omega^*)}ds.
\end{aligned}
\end{equation*}
\medskip
Hence, since the function on the right-hand side is increasing in $t$ and $t$ is arbitrary, it follows that 
\begin{equation*}
\max_{0\leq s\leq t}\|u(s)\|^2_{L^2(\Omega^*)}+
\lambda^{-1}\iint_{\Omega^*\times(0,t)}|\nabla u|^2\,dxds
+C_0\iint_{\Omega^*\times(0,t)}u^2\,dxds\\
\leq \bar C \int_0^t \max_{0\leq \tau\leq s}\|u(\tau)\|^2_{L^2(\Omega^*)}ds,
\end{equation*}
that is
\begin{equation*}
\max_{0\leq s\leq t}\|u(s)\|^2_{L^2(\Omega^*)}
\leq \bar C \int_0^t \max_{0\leq \tau\leq s}\|u(\tau)\|^2_{L^2(\Omega^*)}ds.
\end{equation*}
Finally, using Gronwall's inequality to the function $h(t):=\max_{0\leq s\leq t}\|u(s)\|^2_{L^2(\Omega^*)}$ it follows that $h(t)\leq 0$ which implies $\max_{0\leq s\leq t}\|u(s)\|^2_{L^2(\Omega^*)}=0$ which gives that 
$$u(x,s)=0, \forall (x,s)\in \Omega^*\times (0,t]$$ 
and since $t$ is an arbitrary value in $(0,t^*]$  the claim follows.
\end{proof}
\subsection{Proof of Theorem 3.1}
\begin{proof}
Let $u:=u_1-u_2$ and $w=w_1-w_2$ and let $G$ be the connected component of $\Omega\backslash (D_1\cup D_2)$ containing $\Sigma$. Then, using the form of the nonlinearities $f$ and $g$ we have that
\begin{equation}
\left\{
\begin{aligned}
\partial_t u - div(K_0\nabla u) + a_1u+a_2w &= 0 \qquad &\text{in } \Omega\backslash G\times (0,T), \\
u=K_0 \nabla u \cdot \normal &= 0 \qquad &\text{on }  \Sigma\times (0,T), \\
\partial_t w + a_3u+\epsilon w &= 0 \qquad &\text{in } G\times (0,T) , \\
u(\cdot,0) = 0 \qquad w(\cdot,0) &=0 \qquad &\text{in } G
\end{aligned}
\right.
\label{probcavlin}
\end{equation}
By the uniform estimates of solutions to the problem (\ref{probcav}) we have
\begin{equation}\label{uniformbounds}
|a_i|\leq C_0\,\,\textrm{ in }  G\times [0,T],\quad i=1,2,3.
\end{equation}
Hence, 
\begin{eqnarray}
|\partial_t w|&\leq& C(|u|+|w|),\\
w(\cdot,0)&=&0 \textrm{ in }G.
\end{eqnarray}
and we obtain straightforwardly 
\[
| w(x,t)|\leq \left|\int_0^t\partial_sw(x,s)ds\right|\leq 
C\left(\int_0^t|u(x,s)|ds+\int_0^t|w(x,s)|ds\right)
\]
which implies
\[
\left(| w(x,t)|-C\int_0^t|w(x,s)|ds\right)e^{-Ct}\leq
Ce^{-Ct}\int_0^t|u(x,s)|ds.
\]
Hence
\[
\frac{d}{dt}\left[e^{-Ct}\int_0^t|w(x,s)|ds\right]\leq Ce^{-Ct}\int_0^t|u(x,s)|ds
\]
and integrating it follows that
\[
e^{-Ct}\int_0^t|w(x,s)|ds\leq C\int_0^t\left(e^{-C\eta}\int_0^{\eta}|u(x,s)|ds\right)d\eta
\]
which finally gives
\[
\int_0^t|w(x,s)|ds\leq C\int_0^t\left(e^{C(t-\eta)}\int_0^{\eta}|u(x,s)|ds\right)d\eta\leq (e^{CT}-1)\int_0^{t}|u(x,s)|ds=\bar{C}\int_0^{t}|u(x,s)|ds.
\]
Hence
\[
| w(x,t)|\leq 
C\left(\int_0^t|u(x,s)|ds+\int_0^t|w(x,s)|ds\right)\leq 
(C+C\bar{C})\int_0^t|u(x,s)|ds
\]
and inserting this last inequality and (\ref{uniformbounds}) into the first equation in (\ref{probcavlin}) it follows that
\begin{equation}\label{ineq}
 \left|\partial_t u - div(K_0\nabla u)\right|\leq \Lambda_0\left(|u|+\Lambda_1\int_0^t|u(x,s)|ds\right)
\end{equation}
with $u=u=K_0 \nabla u \cdot \normal = 0 \text{ on }  \Sigma\times (0,T)$ and $u(x,0)=0$ in $G$.
In order to apply the Three Cylinder Inequality (\ref{TC}) we now follow the approach used in \cite{Vessella}. We consider an interior point $P\in \Sigma$ that up to a rigid motion can be fixed to be the origin and such that the unit outward normal vector points in the direction of $-e_3=(-1,0,0)$. Furthermore, we assume that $\partial\Omega\cap Q_{r_0,2M_0}\subset\Sigma$. Note that by the regularity of $u$ it follows that $u\in H^{2,1}((Q_{r_0,2M_0}\cap\Omega )\times (0,T))$. We start extending $u=0$ for $t\leq 0$. Next, we consider the function
$$
\bar{u}=
\begin{cases}
u \text{ in }  (Q_{r_0,2M_0}\cap\Omega)\times (-\infty,T)\\
0 \text{ in } Q_{r_0,2M_0}\backslash (Q_{r_0,2M_0}\cap\Omega)\times (-\infty, T).
    \end{cases}
    $$
    Then since the Cauchy data of $u$ are zero on $\partial\Omega\cap Q_{r_0,2M_0}$ it follows straightforwardly that $\bar u\in H^{2,1}(Q_{r_0,2M_0}\times (-\infty,T))$ and satisfies (\ref{ineq}). We now pick up the ball $B_r(-re_3)$ with $r\in (0,\mu_0r_0)$ and $\mu_0=\min(M_0,\frac{1}{M_0})$ in such a way that $B_r(-re_3)\subset Q_{r_0,2M_0}\backslash ( Q_{r_0,2M_0}\cap\Omega)$ and is tangent to $\Sigma$ at the origin. 
    As a consequence, $\bar u=0$ in $B_r(-re_3)$. Applying the Three Cylinder Inequality to $\bar u$ and reminding that $u\in C^{2+\alpha,1+\alpha/2}$ we finally have that $u=0$ in $(B_{2r}(-re_3)\cap\Omega)\times(0,T]$ and iterating (\ref{TC}) we obtain that $u=0$ in $G\times [0,T-\delta]$ for $\delta\in (0,T)$. 
Let us now argue by contradiction. Assume that $D_1\neq D_2$. 
Let $\tilde{G}=\Omega\setminus G$ and observe that:
$\tilde{G}$ is closed, $\tilde{G}\supseteq D_1\cup D_2$ and

\begin{equation}\label{tildGbd}
\partial\tilde{G}=\big (\partial D_1\cup \partial D_2\big )\cap \partial G\,.
\end{equation}

 Let $\tilde{D}$ be a connected component of $\tilde{G}\setminus D_2$
 (note that $\tilde{D}=D_1$ if $D_1\cap D_2=\emptyset$). Then we have

\begin{equation}\label{tildDbd}
\partial\tilde D\subseteq \partial\big (\tilde{G}\setminus D_2 \big )\subseteq\partial\tilde G\cup\partial D_2\,.
\end{equation}

We further note that, unless $D_1\subset D_2$, we may assume that $\tilde{D}$ contains a subset (of $D_1$) with nonempty interior. Otherwise, we just exchange the roles of $D_1$ and $D_2$.

\smallskip
Let us now define
$\Gamma_1\equiv\partial\tilde{D}\cap\partial D_1\cap\partial\tilde{G}$ and let $\partial\tilde{D}=\Gamma_1\cup \Gamma_2$. Observe that (\ref{tildGbd}) implies that $\Gamma_1\subset\partial D_1\cap \partial G$ and from (\ref{tildDbd}) $\Gamma_2\equiv\partial\tilde{D}\setminus\Gamma_1\subset\partial D_2$ including possibly the case where $\Gamma_2=\emptyset$.
Then, for any choice of the nonlinearity in the system we have that $u_2$ satisfies the following equation
\begin{equation}\label{nonlocaleq}
\partial_t u_2 - \textrm{div}(K_0\nabla u_2) + f_1(u_2) +k_1\int_0^te^{-c_1(t-s)}g_1(u_2)ds = 0 \qquad \text{in } \tilde{D}\times (0,T)
\end{equation}
where $f_1(u)=Au(u-a)(u-1)$ and with the boundary and initial conditions 
\begin{equation}\label{conditions}
K_0 \nabla u_2 \cdot {\nu} = 0 \quad \text{on }\partial\tilde{D}\times (0,T),\qquad u_2(\cdot,0) = 0  \quad \text{in } \tilde{D}.
\end{equation}
In fact, (\ref{nonlocaleq}) holds with $c_1=\epsilon\gamma$ in the Fitzhugh-Nagumo and Rogers-McCulloch models, $c_1=\epsilon$ in the Aliev-Panfilov model, $k_1=1$ in Fitzhugh-Nagumo model and $k_1=u$ in the Aliev-Panfilov and Rogers-McCulloch models, $g_1(u)=\epsilon u$ in the Fitzhugh-Nagumo and Rogers-McCulloch models and $g_1(u)=\varepsilon Au(u-1-a)$.

In fact, we have that $u_2$ and $w_2$ solve the system
\begin{equation}
\left\{
\begin{aligned}
\partial_t u_2 - div(K_0\nabla u_2) + f_1(u_2)+k_1w &= 0 \qquad &\text{in } \tilde{D}\times (0,T), \\
K_0 \nabla u_2 \cdot \normal &= 0 \qquad &\text{on }  \partial\tilde{D}\times (0,T), \\
\partial_t w_2 + c_1w_2 &= g_1(u_2) \qquad &\text{in } \tilde{D}\times (0,T) , \\
u_2(\cdot,0) =0 \qquad w_2(\cdot,0) &= 0 \qquad &\text{in } \tilde{D}
\end{aligned}
\right.
%\label{probcav1}
\end{equation}
Then integrating the equation for $w_2$ we get
$$
w_2(x,t)=\int_0^te^{-c_1(t-s)}g_1(u_2)ds
$$
and plugging it in the first partial differential equation we get exactly (\ref{nonlocaleq}). 
let us now verify that the assumptions of Lemma \ref{energymethod} are satisfied. In fact, using the fact that 
%we have that $f_1(u)=a_1u$,  $f_1'(u_2)=A(3u_2^2-2(a+1)u_2+a)\geq -\frac{A}{3}[\frac{3}{4}+(a-\frac{1}{2})^2]=-\Lambda_{f_1}$. 
%Furthermore, using the fact that 
$|u_2|\leq C$ in $\overline{\tilde{D}}\times [0,T]$, we have that $f_1(u_2)=Au_2(u_2-a)(u_2-1)=a_1u_2$ with $|a_1|\leq A(C+1)(C+a)$,  $|k_1|\leq C+1$  and we can pick up $C_0=(C+1)\max\{A(C+a),1\}$. Finally $g_1(0)=0$ and again by the estimates for $u_2$ we get $|g_1(u_2)|\leq \epsilon\max (1, A(2C+1-a))=C_1$.

%\begin{equation*}
%\left\{
%\begin{aligned}
%\partial_t u_2 - div(K_0\nabla u_2) +qu_2 &= 0 \quad &\text{in } \tilde{D}\times (0,T), \\
%K_0 \nabla u_2 \cdot {\bf n} &= 0 \quad &\text{on }\partial\tilde{D}\times (0,T),\\ 
%u_2(\cdot,0)& = 0  \quad &\text{in } \tilde{D}
%\end{aligned}
%\right.
%\label{probcavlin}
%\end{equation*}
%where the Neumann condition has to be understood in a weak sense. 
%By Theorem \ref{th:1} we have that $q$ is uniformly bounded in $\tilde{D}\times (0,T)$. 
So,  since $u_2 \in C^{2+\alpha,1+\alpha/2}(\overline{\tilde{D}}\times [0,T])$, we can apply Lemma \ref{energymethod} to conclude that $u_2\equiv0$ on $\tilde D\times(0,T-\delta)$ for some $0<\delta,T$
and, again by the unique continuation property,  
$u_2\equiv0$ in $(\Omega\backslash D_2)\times (0,T-\delta)$. Hence $u_2(x,0)=u_0(x)=0$ for all $x\in \Omega\backslash D_2$ contradicting the assumption $u_0\neq 0$. 
\end{proof} 
\section*{Acknowledgments}
This research has been partially performed in the framework of MIUR-PRIN Grant 2020F3NCPX {\it Mathematics for industry 4.0 (Math4I4)}, and MIUR-PRIN-20227HX33Z – “Pattern formation in nonlinear phenomena”.
Andrea Aspri, Elisa Francini, Dario Pierotti and Sergio Vessella are members of the group GNAMPA (Gruppo Nazionale per l’Analisi Matematica, la Probabilità e le loro Applicazioni) of INdAM (Istituto Nazionale di Alta Matematica).

\bibliographystyle{plain}
\bibliography{bibliography}

\begin{thebibliography}{10}

\bibitem{art:BCR20}
E.~Beretta, C.~Cavaterra, and L.~Ratti.
\newblock On the determination of ischemic regions in the monodomain model of
  cardiac electrophysiology from boundary measurements.
\newblock {\em Nonlinearity}, 33(11):5659--5685, 2020.

\bibitem{art:BCPR}
E.~Beretta, M.~C. Cerutti, D.~Pierotti, and L.~Ratti.
\newblock On the reconstruction of cavities in a nonlinear model arising from
  cardiac electrophysiology.
\newblock {\em ESAIM Control Optim. Calc. Var.}, 29:Paper No. 36, 33, 2023.

\bibitem{art:bcmp}
E.~Beretta, M.C. Cerutti, A.~Manzoni, and D.~Pierotti.
\newblock An asymptotic formula for boundary potential perturbations in a
  semilinear elliptic equation related to cardiac electrophysiology.
\newblock {\em Math. Mod. and Meth. in Appl. S.}, 26(04):645--670, 2016.

\bibitem{art:BMR}
E.~Beretta, A.~Manzoni, and L.~Ratti.
\newblock A reconstruction algorithm based on topological gradient for an
  inverse problem related to a semilinear elliptic boundary value problem.
\newblock {\em Inverse Problems}, 33(3):035010, 2017.

\bibitem{art:BRV}
E.~Beretta, L.~Ratti, and M.~Verani.
\newblock Detection of conductivity inclusions in a semilinear elliptic problem
  arising from cardiac electrophysiology.
\newblock {\em Communications in Mathematical Sciences}, 16(7):1975--2002,
  2018.

\bibitem{BCPcav}
Elena Beretta, M.~Cristina Cerutti, and Dario Pierotti.
\newblock On a nonlinear model in domains with cavities arising from cardiac
  electrophysiology.
\newblock {\em Inverse Problems}, 38(10):Paper No. 105005, 16, 2022.

\bibitem{Vessella}
B.~Canuto, E.~Rosset, and S.~Vessella.
\newblock Quantitative estimates of unique continuation for parabolic equations
  and inverse initial-boundary value problems with unknown boundaries.
\newblock {\em Trans. Amer. Math. Soc.}, 354(2):491--535, 2002.

\bibitem{book:pavarino}
P.~Colli~Franzone, L.F. Pavarino, and S.~Scacchi.
\newblock {\em Mathematical Cardiac Electrophysiology}, volume~13 of {\em
  MS\&A}.
\newblock Springer, 2014.

\bibitem{book:escvess}
L.~Escauriaza and S.~Vessella.
\newblock Optimal three cylinder inequalities for solutions to parabolic
  equations with lipschitz leading coefficients.
\newblock In {\em Inverse Problems: Theory and Applications}, pages 79--88.
  American Mathematical Society, 2003.

\bibitem{EvansGariepy}
L.~C. Evans and R.~F. Gariepy.
\newblock {\em Measure theory and fine properties of functions}.
\newblock Textbooks in Mathematics. CRC Press, Boca Raton, FL, revised edition,
  2015.

\bibitem{art:FPLHMDQdB}
A.~Frontera, S.~Pagani, L.~R. Limite, A.~Hadjis, A.~Manzoni, L.~Dedé,
  A.~Quarteroni, and P.~Della~Bella.
\newblock Outer loop and isthmus in ventricular tachycardia circuits:
  Characteristics and implications.
\newblock {\em Heart rhythm}, 17:1719--1728, 2020.

\bibitem{art:PSIRBF}
A.~Lopez-Perez, R.~Sebastian, M.~Izquierdo, R.~Ruiz, M.~Bishop, and J.~M.
  Ferrero.
\newblock Personalized cardiac computational models: From clinical data to
  simulation of infarct-related ventricular tachycardia.
\newblock {\em Frontiers in Physiology}, 10, 2019.

\bibitem{lunardi}
A.~Lunardi.
\newblock {\em Analytic semigroups and optimal regularity in parabolic
  problems}.
\newblock Modern Birkh\"{a}user Classics. Birkh\"{a}user/Springer Basel AG,
  Basel, 1995.
\newblock [2013 reprint of the 1995 original].

\bibitem{book:pao}
C.~V. Pao.
\newblock {\em Nonlinear parabolic and elliptic equations}.
\newblock Plenum Press, New York, 1992.

\bibitem{art:RCSRGDRRA}
J.~Relan, P.~Chinchapatnam, M.~Sermesant, K.~Rhode, M.~Ginks, H.~Delingette,
  C.~A. Rinaldi, R.~Razavi, and N.~Ayache.
\newblock Coupled personalization of cardiac electrophysiology models for
  prediction of ischaemic ventricular tachycardia.
\newblock {\em Interface focus}, 1:396--407, 2011.

\bibitem{book:sundes-lines}
J.~Sundnes, G.~T. Lines, X.~Cai, B.~F. Nielsen, K.A. Mardal, and A.~Tveito.
\newblock {\em Computing the electrical activity in the heart}.
\newblock Monographs in Computational Science and Engineering Series, Volume 1.
  Springer, 2006.

\bibitem{art:Vessella2003}
S.~Vessella.
\newblock Carleman estimates, optimal three cylinder inequality, and unique
  continuation properties for solutions to parabolic equations.
\newblock {\em Communications in Partial Differential Equations}, 28(3-4):637
  -- 676, 2003.

\end{thebibliography}

\end{document}